\documentclass[reqno]{amsart}
\usepackage{color}
\usepackage{amsmath}
\usepackage{amsfonts}
\usepackage{amssymb}
\usepackage{amsthm}
\usepackage{newlfont}
\usepackage{graphicx}
\usepackage{hyperref}

\newtheorem{thm}{Theorem}[section]

\newtheorem{lem}[thm]{Lemma}
\newtheorem{prop}[thm]{Proposition}
\theoremstyle{definition}
\newtheorem{defn}[thm]{Definition}
\theoremstyle{remark}
\newtheorem{rem}[thm]{Remark}

\numberwithin{equation}{section}

\newcommand{\eps}{\varepsilon}

\newcommand{\lsm}{\lesssim}

\newcommand{\C}{{\mathbb{C}}}
\newcommand{\R}{{\mathbb{R}}}

\let\Im=\undefined\DeclareMathOperator{\Im}{Im}

\newcounter{smalllist}
\newenvironment{SL}{\begin{list}{{$($\roman{smalllist}\/$)$\hss}}{%
\setlength{\topsep}{0mm}\setlength{\parsep}{0mm}\setlength{\itemsep}{0mm}%
\setlength{\labelwidth}{2.0em}\setlength{\itemindent}{2.5em}\setlength{\leftmargin}{0em}\usecounter{smalllist}%
}}{\end{list}}

\newcommand{\bit}{\noindent$\bullet$ }

\title[The focusing mass-critical NLS]{Characterization of minimal-mass blowup solutions to the focusing mass-critical NLS}
\author{Rowan Killip}
\address{University of California, Los Angeles, and Institute for Advanced Study, Princeton, NJ}
\author{Dong Li}
\address{Institute for Advanced Study, Princeton, NJ}
\author{Monica Visan}
\address{Institute for Advanced Study, Princeton, NJ}
\author{Xiaoyi Zhang}
\address{Academy of Mathematics and System Sciences, Beijing, and Institute for Advanced Study, Princeton, NJ}


\begin{document}

\maketitle

\begin{abstract}
Let $d\geq 4$ and let $u$ be a global solution to the focusing mass-critical nonlinear Schr\"odinger equation
$iu_t+\Delta u=-|u|^{\frac 4d}u$ with spherically symmetric $H_x^1$ initial data and mass equal to that of the ground state $Q$.
We prove that if $u$ does not scatter then, up to phase rotation and scaling, $u$ is the solitary wave $e^{it}Q$.
Combining this result with that of Merle \cite{merle2}, we obtain that in dimensions $d\geq 4$, the only spherically symmetric
minimal-mass non-scattering solutions are, up to phase rotation and scaling, the pseudo-conformal ground state and the ground state solitary wave.
\end{abstract}

\section{Introduction}

We consider the focusing mass-critical nonlinear Schr\"odinger equation
\begin{equation}\label{nls}
iu_t+\Delta u=-|u|^{\frac 4d}u
\end{equation}
in dimensions $d\geq 4$; here $u(t,x)$ is a complex-valued function on $\mathbb R\times \R^d$.

The name ``mass-critical'' refers to the fact that the scaling symmetry
\begin{equation}\label{scaling}
u(t,x)\mapsto u_{\lambda}(t,x)=\lambda^{\frac
d2}u(\lambda^{2}t,\lambda x)
\end{equation}
leaves both the equation and the mass invariant. The mass of a solution is defined as
\begin{equation}\label{mass}
M(u(t)):=\int_{\mathbb R^d}|u(t,x)|^2 \, dx,
\end{equation}
and is conserved under the flow (see Theorem~\ref{T:local} below).

In this paper, we investigate the Cauchy problem for \eqref{nls} with spherically symmetric initial data.
Before describing our results, we review some background material.  We begin by making the notion of a solution more precise:

\begin{defn}[Solution]\label{D:solution}
A function $u: I \times \R^d \to \C$ on a non-empty time interval $I \subset \R$ (possibly infinite or semi-infinite)
is a \emph{strong $L^2_x(\R^d)$ solution} (or \emph{solution} for short) to \eqref{nls} if it lies in the class
$C^0_t L^2_x(K \times \R^d) \cap L^{2(d+2)/d}_{t,x}(K \times \R^d)$ for all compact $K \subset I$  and obeys the Duhamel formula
\begin{align}\label{old duhamel}
u(t_1) = e^{i(t_1-t_0)\Delta} u(t_0) + i \int_{t_0}^{t_1} e^{i(t_1-t)\Delta} \bigl(|u|^{\frac 4d}u\bigr)(t)\, dt
\end{align}
for all $t_0, t_1 \in I$.  We refer to the interval $I$ as the \emph{lifespan} of $u$. We say that $u$ is a \emph{maximal-lifespan solution}
if the solution cannot be extended to any strictly larger interval. We say that $u$ is a \emph{global solution} if $I = \R$.
\end{defn}

The condition that $u$ belongs to $L_{t,x}^{2(d+2)/d}$ locally in time is natural for several reasons.
From the Strichartz estimate (see Lemma \ref{L:strichartz}), we see that solutions to the linear equation lie in this space.
Moreover, the existence of solutions that belong to this space is guaranteed by the local theory (see Theorem~\ref{T:local} below).
This condition is also necessary in order to ensure uniqueness of solutions.
Solutions to \eqref{nls} in this class have been intensively studied; see, for example,
\cite{BegoutVargas, cwI, hmidi-keraani, keraani, ktv:2d, kvz:blowup, merle1, merle2, compact, weinstein, weinstein:charact}
and the many references within.

Associated to our notion of solution is a corresponding notion of blowup.  As demonstrated by Theorem \ref{T:local} below,
this corresponds precisely to the impossibility of continuing the solution or to the absence of scattering:

\begin{defn}[Blowup]\label{D:blowup}
We say that a solution $u$ to \eqref{nls} \emph{blows up forward in time} if there exists a time $t_0 \in I$ such that
$$ \int_{t_0}^{\sup I} \int_{\R^d} |u(t,x)|^{2(d+2)/d}\, dx \,dt = \infty$$
and that $u$ \emph{blows up backward in time} if there exists a time $t_0 \in I$ such that
$$ \int_{\inf I}^{ t_0} \int_{\R^d} |u(t,x)|^{2(d+2)/d}\, dx\, dt = \infty.$$
We emphasize once again that with this terminology, a solution blows up if and only if it fails to scatter.  This could result, for example,
from the divergence of the kinetic energy (as must occur for a finite-time blowup solution with $H^1_x$ data) or from soliton-like behaviour.
\end{defn}

The local theory for \eqref{nls} was worked out by Cazenave and Weissler \cite{cwI}; we record their results below:

\begin{thm}[Local wellposedess, \cite{cwI, caz}] \label{T:local}
Given $u_0\in L_x^2(\R^d)$ and $t_0\in \R$, there exists a unique maximal-lifespan solution $u$ to \eqref{nls} with $u(t_0)=u_0$.
Let $I$ denote the maximal-lifespan.  Then,

\bit (Local existence) $I$ is an open neighborhood of $t_0$.

\bit (Mass conservation) The solution obeys $M(u(t))=M(u_0)$.

\bit (Blowup criterion) If\/ $\sup I$ or $\inf I$ are finite, then $u$ blows up in the corresponding time direction.

\bit (Continuous dependence) The map that takes initial data to the corresponding strong solution is uniformly continuous
on compact time intervals for bounded sets of initial data.

\bit (Scattering) If\/ $\sup I=\infty$ and $u$ does not blow up forward in time, then $u$ scatters forward in time,
that is, there exists a unique $u_+\in L_x^2(\R^d)$ such that
\begin{align}\label{like u+}
\lim_{t\to\infty}\|u(t)-e^{it\Delta}u_+\|_2= 0.
\end{align}
Conversely, given $u_+ \in L^2_x(\R^d)$ there is a unique solution to \eqref{nls} in a neighbourhood of infinity
so that \eqref{like u+} holds.  Analogous statements hold in the negative time direction.

\bit (Small data global existence) If $M(u_0)$ is sufficiently small depending on $d$, then $u$ is a global solution
with finite $L_{t,x}^{2(d+2)/d}$-norm.
\end{thm}

By Theorem \ref{T:local}, all solutions with sufficiently small mass are global and scatter both forward and backward in time.
However, solutions with large mass may blow up; indeed, the existence of finite-time blowup solutions was proved by Glassey \cite{glassey}.
Moreover, there exist explicit examples of blowup solutions. Two typical examples are produced as follows:
Let $Q$ denote the ground state, that is, the unique positive radial Schwartz solution to the elliptic equation
\begin{equation}\label{ground state}
\Delta Q+Q^{1+\frac 4d}=Q.
\end{equation}
The existence and uniqueness of $Q$ was established by Berestycki and Lions \cite{blions} and Kwong \cite{kwong}, respectively.  Then
\begin{equation}\label{soliton}
u(t,x):=e^{it}Q(x)
\end{equation}
is a global solution to \eqref{nls}, which blows up both forward and backward in time in the sense of Definition~\ref{D:blowup}.
Moreover, by applying the pseudo-conformal transformation to $u$, we obtain another solution to \eqref{nls},
\begin{equation}\label{psu}
v(t,x):=|t-T|^{-\frac d2}e^{i\frac{|x|^2-4}{4(t-T)}}Q\bigl(\tfrac x{t-T}\bigr),
\end{equation}
which blows up at the finite time $T$.  A simple calculation shows that $M(v)=M(Q)$.

It is widely believed that up to the symmetries of \eqref{nls}, these examples are the only minimal-mass obstructions to global well-posedness
and scattering.  Recently, this has been verified in the spherically symmetric case in dimensions $d\ge 2$.

\begin{thm}[Well-posedness and scattering below $M(Q)$, \cite{ktv:2d,kvz:blowup}]\label{T:scatter}
\ Let $d\ge 2$ and let $u_0\in L_x^2$ be spherically symmetric with $M(u_0)< M(Q)$.  Then there exists a unique global
solution $u$ to \eqref{nls} with initial data $u_0$.  Moreover, $u$ scatters in both time directions.
If in addition $u_0\in H^1_x$, then $u \in L_t^\infty H^1_x$ and scattering holds in the $H^1_x$ topology.
\end{thm}

We remind the reader of an important related result of Weinstein, \cite{weinstein}:  Any initial data $u_0\in H^1_x$ with $M(u_0)<M(Q)$
leads to a global solution.  Note that this holds without any symmetry assumptions.  However, since the global solution is constructed
by iterating a local existence result, one obtains no information on the long-time behaviour of the solution.  In particular, scattering is not proved;
indeed, scaling arguments suggest that scattering for $H^1_x$ solutions is as hard as for general $L_x^2$ solutions.  Let us note however,
that combining Weinstein's result with the pseudo-conformal transformation yields scattering for initial data in
$\Sigma:=\{f\in H^1_x : \, |x|f\in L_x^2\}$ with mass less than that of the ground state.

According to Theorem~\ref{T:scatter}, \eqref{soliton} and \eqref{psu} are two examples of spherically symmetric
\emph{minimal-mass blowup solutions}.  It is then natural to ask if there are any other such examples.
In this paper, we will give a negative answer to this question under some constraints (see Theorem~\ref{T:main} below).

The characterization of minimal-mass blowup solutions was initiated by Weinstein \cite{weinstein:charact}, who showed the following:
Let $u$ be an $H_x^1$ solution with minimal mass that blows up in finite time; then there exist functions $\theta(t), x(t), \lambda(t)$
so that
\begin{align}\label{weinstein's result}
\lambda(t)^{\frac d2} e^{i\theta(t)} u\bigl(t, \lambda(t)x+x(t)\bigr) \to Q \quad \text{in} \quad H^1_x
\end{align}
as $t$ approaches the blowup time.  In truth, convergence follows along any sequence of times for which the kinetic energy diverges;
for an $H_x^1$ solution that blows up in finite time, any sequence converging to the blowup time has this property.

Merle \cite{merle1, merle2} extended this result to show that if an $H_x^1$ solution with minimal mass blows up in finite time,
then it must be \emph{equal} to the pseudo-conformal ground state \eqref{psu}, up to the symmetries of the equation (that is,
phase rotation, space translation, scaling, and Galilei boosts). The proof, which was later simplified by Hmidi and Keraani
\cite{hmidi-keraani}, relies heavily on the finiteness of the blowup time. Note that by using the pseudo-conformal symmetry,
this result immediately implies that solutions belonging to $\Sigma:=\{f\in H^1_x : \, |x|f\in L_x^2\}$ that have $M(u)=M(Q)$
and blow up in infinite time must be the solitary wave \eqref{soliton}, up to the symmetries of the equation.

This leaves open the problem of characterizing general (i.e.\ non-$\Sigma$) minimal-mass $H^1_x$ solutions which fail to scatter.
In this paper, we settle this problem in dimension $d\ge 4$ in the spherically symmetric case.
We will show that, up to phase rotation and scaling, the only such solution that blows up in infinite time
is the solitary wave \eqref{soliton}.  Combining this result with \cite{merle2} leads to

\begin{thm}[Characterization of the blowup profile] \label{T:main}
Let $d\geq 4$ and let $u_0 \in H_x^1 (\mathbb R^d)$ be spherically symmetric and such that $M(u_0) = M(Q)$.
Let $u: I\times \mathbb R^d\to \C$ be the maximal-lifespan solution to \eqref{nls} with prescribed initial data
$u(t_0)=u_0$ at time $t_0\in I$.  Assume that $u$ blows up in the sense of Definition~\ref{D:blowup} in at least one time direction.
Then either $I$ is semi-infinite (with finite endpoint $T$) and
\begin{align}\label{E:ft case}
u(t,x) = e^{i\theta_0}\lambda_0^{\frac d2} |t-T|^{-\frac d2} e^{ \frac{i|x|^2}{4(t-T)} - \frac {i} {t-T}}  Q\bigl(\tfrac {\lambda_0 x}{t-T}\bigr),
\end{align}
or $I=\R$ and
\begin{align}\label{E:it case}
u(t,x)=e^{i\theta_0} e^{it} \lambda_0^{\frac d2} Q(\lambda_0 x)
\end{align}
for some parameters $\theta_0\in[0,2\pi)$ and $\lambda_0>0$.
\end{thm}

In some sense, the study of infinite-time blowup solutions lies half-way between the study of finite-time blowup
solutions and the study of global well-posedness.  As such, this paper is something of a hybrid between \cite{hmidi-keraani,merle2}
and \cite{ktv:2d,kvz:blowup}.

The restriction to spherically symmetric data ultimately stems from the fact that Theorem~\ref{T:scatter} is not known without this
assumption.  The reason we need to assume $d\geq 4$ is more technical; without it we are unable to control the behaviour of the
kinetic energy in a satisfactory way.  See Theorem~\ref{T:kinetic} and the discussion that follows it.

\subsection{Outline of the proof}

As in previous investigations, the role of the hypothesis $u_0\in H^1_x$ (as opposed to $u_0\in L^2_x$) is to gain access to
an additional conservation law, namely, the energy:
\begin{equation}\label{E:Energy}
E(u(t)):=\int_{\R^d} \tfrac 12 |\nabla u(t,x)|^2-  \tfrac d{2(d+2)}|u(t,x)|^{\frac{2(d+2)}d} \,dx.
\end{equation}

The next proposition, which is due to Weinstein, demonstrates the important relationship between $Q$ and the energy:

\begin{prop}[Sharp Gagliardo--Nirenberg inequality, \cite{weinstein}]\label{P:variational}
For $f\in H_x^1(\mathbb R^d)$,
\begin{equation}\label{sharp-gn}
\|f\|_{\tfrac{2(d+2)}d}^{\tfrac{2(d+2)}d}\le \frac{d+2}d\left(\frac{\|f\|_2}{\|Q\|_2} \right)^{\frac 4d}\|\nabla f\|_2^2,
\end{equation}
with equality if and only if
\begin{equation}\label{GNeq}
f(x)= c e^{i\theta_0}\lambda_0^{\frac d2}Q(\lambda_0 (x-x_0))
\end{equation}
for some $\theta\in[0,2\pi)$, $x_0\in\R^d$, and $c,\lambda_0\in(0,\infty)$.  In particular,
if $M(f)=M(Q)$, then $E(f)\geq 0$ with equality if and only \eqref{GNeq} holds with $c=1$.
\end{prop}

In addition to being a conserved quantity, the energy plays a further important role due to its appearance in
a monotonicity formula known as the virial identity:
\begin{equation}\label{true virial}
\partial_{tt} \int_{\R^d} |x|^2 |u(t,x)|^2\, dx = 16 E(u).
\end{equation}
This identity, together with minor modifications, has been a corner-stone in the investigation of blow-up
solutions since Glassey \cite{glassey} used it to demonstrate that negative energy solutions with initial data in $\Sigma$ must blow up.

The proof of Theorem~\ref{T:main} relies on both the sharp Gagliardo-Nirenberg inequality and a truncated version of the virial identity.
Recall that we need only consider the case where $u$ blows up in infinite time, since Merle's result \cite{merle2} covers the
other case.  As $M(u)=M(Q)$, Proposition~\ref{P:variational} shows that $E(u)\geq 0$ with equality if and only if \eqref{E:it case}
holds.  Thus we may prove the theorem by ruling out the possibility that $E(u)>0$; this will be done using a truncation of the
virial identity.

Without loss of generality, we may assume that $u$ blows up at infinity forward in time.  The key to using the virial identity
in our context is proving that the mass and the kinetic energy remain concentrated near the spatial origin.  For the mass, we
may rely on \cite{BegoutVargas,keraani-2,compact} together with the identification of $M(Q)$ as the minimal mass, which was
done in \cite{ktv:2d,kvz:blowup}.  The precise result we require reads as follows:

\begin{thm}[Almost periodicity modulo scaling]\label{T:periodic}
Let $u:[t_0,\infty) \times \mathbb R^d\to\C$ be a spherically symmetric solution to \eqref{nls} which satisfies $M(u)=M(Q)$
and blows up forward in time.  Then $u$ is \emph{almost periodic modulo scaling} in the following sense:
there exist functions $N:[t_0,\infty)\to \mathbb R^+$ and $C:\mathbb R^+\to\mathbb R^+$ such that
\begin{equation}\label{eq_mass_local}
\int_{|x|\ge C(\eta)/N(t)}|u(t,x)|^2 dx\le \eta \quad \text{and} \quad  \int_{|\xi|\ge C(\eta)N(t)}|\hat u(t,\xi)|^2 d\xi\le \eta
\end{equation}
for all $t\in [t_0, \infty)$ and $\eta>0$.  We refer to the function $N$ as the frequency scale function
and to $C$ as the compactness modulus function.
\end{thm}

\begin{rem}\label{R:bdd N} The parameter $N(t)$ measures the frequency scale of the solution at time $t$ while $1/N(t)$ measures its spatial scale.
Further properties of the function $N(t)$ are discussed in \cite{ktv:2d, compact}.  One such property that we will use
is a consequence of the local-constancy property of $N(t)$ (see \cite[Corollary~3.6]{ktv:2d}), namely,
$$N(t_1)\gtrsim N(t_2) \langle t_1-t_2\rangle^{-1/2},$$
for all pairs $t_1,t_2\in [t_0,\infty)$.
\end{rem}

\begin{rem} By the Ascoli-Arzela Theorem, \eqref{eq_mass_local} is equivalent to saying that the set
$\{N(t)^{-\frac d2}u(t,\tfrac x{N(t)}), \ t\in[t_0,\infty)\}$ is precompact in $L_x^2(\R^d)$.
\end{rem}

One important consequence of the fact that $u$ is almost periodic modulo scaling (near positive infinity) is the following Duhamel formula,
where the free evolution term disappears:

\begin{lem}[{\cite[Section 6]{compact}}]\label{duhamel L}
Let $u$ be an almost periodic solution to \eqref{nls} on  $[t_0,\infty)$.  Then, for all $t\in [t_0,\infty)$,
\begin{equation}\label{duhamel}
u(t) = - \lim_{T\nearrow\,\infty}i\int_t^T e^{i(t-t')\Delta} \bigl(|u|^{\frac 4d}u\bigr)(t')\,dt'
\end{equation}
as a weak limit in $L_x^2$.
\end{lem}

With the localization of mass done for us, the key ingredient in the proof of Theorem~\ref{T:main} is the localization of kinetic energy.
Due to the focusing nature of the problem, both the kinetic energy and the potential energy can become very large while still maintaining
the finiteness of the energy. This makes the localization of the kinetic energy rather surprising.  We will prove the following

\begin{thm}[Kinetic energy localization] \label{T:kinetic}
Let $d\ge 4$ and let $u_0\in H_x^1(\mathbb R^d)$ be spherically symmetric with $M(u_0)=M(Q)$.  Let $u$ be a global solution to \eqref{nls}
with initial data $u(0)= u_0$.  Assume that $u$ is almost periodic modulo scaling on $[0,\infty)$ with frequency scale function $N(t)$.
Then, for any $\eta>0$ there exists $C(\eta)>0$ such that
\begin{align*}
\|\nabla u(t)\|_{L_x^2(|x|>C(\eta)\langle N(t)^{-1}\rangle)} \le \eta.
\end{align*}
\end{thm}

As alluded to earlier, the restriction to dimensions $d\geq 4$ stems from our inability to prove kinetic energy localization (uniformly
as $t\to\infty$) in lower dimensions.  Ultimately, the problem is that knowing only $u\in L^\infty_t L^2_x$, it is impossible to put
the nonlinearity $|u|^{\frac4d}u$ into any space $L^\infty_t L^p_x$ with $p\geq 1$ when $d<4$.  For $d\geq 4$, one easily sees that
$|u|^{\frac4d}u \in L^\infty_t L^{2d/(d+4)}_x$.

To prove Theorem~\ref{T:kinetic}, we make use of a decomposition into incoming and outgoing waves; this serves to minimize the
contribution from the nonlinearity near the origin (where we have only the \emph{a priori} estimate described in the previous
paragraph) and refocuses attention at large radii where we can take advantage of spherical symmetry to obtain smallness. The key
point is to use the Duhamel formula into the future to control the outgoing portion of $u$ and the Duhamel formula into the past
to control the incoming portion. The particular decomposition we use is taken from \cite{ktv:2d,kvz:blowup}; the tool we use to
exploit the spherical symmetry is a weighted Strichartz inequality, Lemma~\ref{L:wes}, which is also taken from these papers.
Section~3 is devoted to the proof Theorem~\ref{T:kinetic}.

Without a technique such as the in/out decomposition, we do not see how to preclude the following dangerous scenario:  there are extremely
high-frequency waves very far from the origin which contribute significant kinetic energy (and so cause trouble with any virial-type
arguments) whilst carrying essentially no mass (and thus not contradicting pre-compactness in $L^2_x$).

The main result, Theorem~\ref{T:main}, is proved in Section~4.  Here the argument breaks into two cases.  When $N(t)$ is bounded from below we
make use of the localization of kinetic energy to run a truncated virial argument and so show that $E(u)=0$.  Secondly,
we show that $N(t)$ cannot converge to zero, even along a subsequence.  This second part of the argument is closely reminiscent
of the treatment of the finite-time blowup case in \cite{hmidi-keraani}.

\subsection*{Acknowledgements} The authors were supported by the National Science Foundation under agreement No. DMS-0635607.
R.~Killip was further supported by NSF grants DMS-0701085 and DMS-0401277.  X.~Zhang was also supported by NSF grant No.~10601060
and project 973 in China.

%
%
%
%

\section{Preliminaries}

\subsection{Some notation}
We write $X \lesssim Y$ or $Y \gtrsim X$ to indicate $X \leq CY$ for some constant $C>0$.  We use $O(Y)$ to denote any quantity $X$
such that $|X| \lesssim Y$.  We use the notation $X \sim Y$ whenever $X \lesssim Y \lesssim X$.  The fact that these constants
depend upon the dimension $d$ will be suppressed.  If $C$ depends upon some additional parameters, we will indicate this with
subscripts; for example, $X \lesssim_u Y$ denotes the assertion that $X \leq C_u Y$ for some $C_u$ depending on $u$.

We use the `Japanese bracket' convention $\langle x \rangle := (1 +|x|^2)^{1/2}$.

We write $L^q_t L^r_{x}$ to denote the Banach space with norm
$$ \| u \|_{L^q_t L^r_x(\R \times \R^d)} := \Bigl(\int_\R \Bigl(\int_{\R^d} |u(t,x)|^r\ dx\Bigr)^{q/r}\ dt\Bigr)^{1/q},$$
with the usual modifications when $q$ or $r$ are equal to infinity, or when the domain $\R \times \R^d$ is replaced by a smaller
region of spacetime such as $I \times \R^d$.  When $q=r$ we abbreviate $L^q_t L^q_x$ as $L^q_{t,x}$.

\subsection{Basic harmonic analysis}\label{ss:basic}
Let $\varphi\in C^\infty(\R^d)$ be a radial bump function supported in the ball $\{ x \in \R^d: |x| \leq \frac{25} {24} \}$ and equal to
one on the ball $\{ x \in \R^d: |x| \leq 1 \}$.  For any constant $C>0$, we denote $\varphi_{\le C}(x):= \varphi \bigl( \tfrac{x}{C}\bigr)$
and $\varphi_{> C}:=1-\varphi_{\le C}$.

For each number $N > 0$, we define the Fourier multipliers
\begin{align*}
\widehat{P_{\leq N} f}(\xi) &:= \varphi_{\leq N}(\xi) \hat f(\xi)\\
\widehat{P_{> N} f}(\xi) &:= \varphi_{> N}(\xi) \hat f(\xi)\\
\widehat{P_N f}(\xi) &:= (\varphi_{\leq N} - \varphi_{\leq N/2})(\xi) \hat f(\xi)
\end{align*}
and similarly $P_{<N}$ and $P_{\geq N}$.  We also define
$$ P_{M < \cdot \leq N} := P_{\leq N} - P_{\leq M} = \sum_{M < N' \leq N} P_{N'}$$
whenever $M < N$.  We will usually use these multipliers when $M$ and $N$ are \emph{dyadic numbers} (that is, of the form $2^n$
for some integer $n$); in particular, all summations over $N$ or $M$ are understood to be over dyadic numbers.  Nevertheless, it
will occasionally be convenient to allow $M$ and $N$ to not be a power of $2$.  As $P_N$ is not truly a projection, $P_N^2\neq P_N$,
we will occasionally need to use fattened Littlewood-Paley operators:
\begin{equation}\label{PMtilde}
\tilde P_N := P_{N/2} + P_N +P_{2N}.
\end{equation}
These obey $P_N \tilde P_N = \tilde P_N P_N= P_N$.

Like all Fourier multipliers, the Littlewood-Paley operators commute with the propagator $e^{it\Delta}$, as well as with
differential operators such as $i\partial_t + \Delta$. We will use basic properties of these operators many many times,
including

\begin{lem}[Bernstein estimates]\label{Bernstein}
 For $1 \leq p \leq q \leq \infty$,
\begin{align*}
\bigl\| |\nabla|^{\pm s} P_N f\bigr\|_{L^p_x(\R^d)} &\sim N^{\pm s} \| P_N f \|_{L^p_x(\R^d)},\\
\|P_{\leq N} f\|_{L^q_x(\R^d)} &\lesssim N^{\frac{d}{p}-\frac{d}{q}} \|P_{\leq N} f\|_{L^p_x(\R^d)},\\
\|P_N f\|_{L^q_x(\R^d)} &\lesssim N^{\frac{d}{p}-\frac{d}{q}} \| P_N f\|_{L^p_x(\R^d)}.
\end{align*}
\end{lem}

While it is true that spatial cutoffs do not commute with Littlewood-Paley operators, we still have the following:

\begin{lem}[Mismatch estimates in real space]\label{L:mismatch_real}
Let $R,N>0$.  Then
\begin{align*}
\bigl\| \varphi_{> R} \nabla P_{\le N} \varphi_{\le\frac R2} f \bigr\|_p &\lsm_m N^{1-m} R^{-m} \|f\|_p\\
\bigl\| \varphi_{> R}  P_{\leq N} \varphi_{\le\frac R2} f \bigr\|_p &\lsm_m N^{-m} R^{-m} \|f\|_p
\end{align*}
for any $1\le p\le \infty$ and $m\geq 0$.
\end{lem}

\begin{proof}
We will only prove the first inequality; the second follows similarly.

It is not hard to obtain kernel estimates for the operator $\varphi_{> R}\nabla P_{\le N}\varphi_{\le\frac R2}$. Indeed, an
exercise in non-stationary phase shows
\begin{align*}
\bigl|\varphi_{> R}\nabla P_{\le N}\varphi_{\le\frac R2}(x,y)\bigr|
\lesssim N^{d+1-2k} |x-y|^{-2k}\varphi_{|x-y|>\frac R2}
\end{align*}
for any $k\geq 0$.  An application of Young's inequality yields the claim.
\end{proof}

Similar estimates hold when the roles of the frequency and physical spaces are interchanged.  The proof is easiest when
working on $L_x^2$, which is the case we will need; nevertheless, the following statement holds on $L_x^p$ for any $1\leq p\leq \infty$.

\begin{lem}[Mismatch estimates in frequency space]\label{L:mismatch_fre}
For $R>0$ and $N,M>0$ such that $\max\{N,M\}\geq 4\min\{N,M\}$,
\begin{align*}
\bigl\|  P_N \varphi_{\le{R}} P_M f \bigr\|_2 &\lsm_m \max\{N,M\}^{-m} R^{-m} \|f\|_2\\
\bigl\|  P_N \varphi_{\le {R}} \nabla P_M f \bigr\|_2 &\lsm_m M \max\{N,M\}^{-m} R^{-m} \|f\|_2.
\end{align*}
for any $m\geq 0$.  The same estimates hold if we replace $\varphi_{\le R}$ by $\varphi_{>R}$.
\end{lem}

\begin{proof}
The first claim follows from Plancherel's Theorem and Lemma~\ref{L:mismatch_real} and its adjoint.  To obtain the second claim from this, we write
$$
P_N \varphi_{\le {R}} \nabla P_M = P_N \varphi_{\le {R}} P_M \nabla \tilde P_M
$$
and note that $\|\nabla \tilde P_M\|_{L_x^2\to L_x^2}\lesssim M$.
\end{proof}

\subsection{Strichartz estimates}
Throughout this section we assume $d\geq 4$; of course, some of the estimates recorded below hold also in lower dimensions,
but we will not need that here.  First, we recall the following standard Strichartz estimate:

\begin{lem}[Strichartz]\label{L:strichartz}  Let $I$ be an interval, $t_0 \in I$, and let $u_0 \in L^2_x(\R^d)$
and $F \in L^{2(d+2)/(d+4)}_{t,x}(I \times \R^d)$.  Then, the function $u$ defined by
$$ u(t) := e^{i(t-t_0)\Delta} u_0 - i \int_{t_0}^t e^{i(t-t')\Delta} F(t')\ dt'$$
obeys the estimate
$$
\|u \|_{L^\infty_t L^2_x} + \| u \|_{L^{\frac{2(d+2)}d}_{t,x}} + \|u \|_{L^2_t L^{\frac{2d}{d-2}}_x}
    \lesssim \| u_0 \|_{L^2_x} + \|F\|_{L^{\frac{2(d+2)}{d+4}}_{t,x}},
$$
where all spacetime norms are over $I\times\R^d$.
\end{lem}

\begin{proof}
See, for example, \cite{gv:strichartz, strichartz}.  For the endpoint see \cite{tao:keel}.
\end{proof}

We will also need a weighted Strichartz estimate, which exploits heavily the spherical symmetry in order to obtain spatial decay.

\begin{lem}[Weighted Strichartz, \cite{ktv:2d, kvz:blowup}]\label{L:wes} Let $I$ be an interval, $t_0 \in I$, and let
$F:I\times\R^d\to \C$ be spherically symmetric.  Then,
$$ \biggl\|\int_{t_0}^t e^{i(t-t')\Delta} F(t')\, dt' \biggr\|_{L_x^2}
\lesssim \bigl\||x|^{-\frac{2(d-1)}q}F \bigr\|_{L_t^{\frac{q}{q-1}}L_x^{\frac{2q}{q+4}}(I \times \R^d)}
$$
for all $4\leq q\leq \infty$.
\end{lem}

\subsection{An in-out decomposition}
We will need an incoming/outgoing decomposition; we will use the one developed in \cite{ktv:2d, kvz:blowup}.
As there, we define operators $P^{\pm}$ by
\begin{align*}
[P^{\pm} f](r) :=\tfrac12 f(r)\pm \tfrac{i}{\pi} \int_0^\infty \frac{r^{2-d}\,f(\rho)\,\rho^{d-1}\,d\rho}{r^2-\rho^2},
\end{align*}
where the radial function $f: \R^d\to \C$ is written as a function of radius only.
We will refer to $P^+$ is the projection onto outgoing spherical waves; however, it is not a true projection as it is neither idempotent
nor self-adjoint.  Similarly, $P^-$ plays the role of a projection onto incoming spherical waves; its kernel is the complex
conjugate of the kernel of $P^+$ as required by time-reversal symmetry.

For $N>0$ let $P_N^{\pm}$ denote the product $P^{\pm}P_N$ where $P_N$ is the Littlewood-Paley projection.
We record the following properties of $P^{\pm}$ from \cite{ktv:2d, kvz:blowup}:

\begin{prop}[Properties of $P^\pm$, \cite{ktv:2d, kvz:blowup}]\label{P:P properties}\leavevmode
\begin{SL}
\item $P^+ + P^- $ represents the projection from $L^2$ onto $L^2_{\text{rad}}$.  In particular, it acts as the identity on radial functions.

\item Fix $N>0$.  Then
$$
\bigl\| \chi_{\gtrsim\frac 1N} P^{\pm}_{\geq N} f\bigr\|_{L^2(\R^d)}
\lesssim \bigl\| f \bigr\|_{L^2(\R^d)}
$$
with an $N$-independent constant.

\item For $|x|\gtrsim N^{-1}$ and $t\gtrsim N^{-2}$, the
integral kernel obeys
\begin{equation*}
\bigl| [P^\pm_N e^{\mp it\Delta}](x,y) \bigr| \lesssim \begin{cases}
    (|x||y|)^{-\frac {d-1}2}|t|^{-\frac 12}  &: \  |y|-|x|\sim  Nt \\[1ex]
     \frac{N^d}{(N|x|)^{\frac{d-1}2}\langle N|y|\rangle^{\frac{d-1}2}}
     \bigl\langle N^2t + N|x| - N|y| \bigr\rangle^{-m}
            &: \  \text{otherwise}\end{cases}
\end{equation*}
for all $m\geq 0$.

\item For $|x|\gtrsim N^{-1}$ and $|t|\lesssim N^{-2}$, the
integral kernel obeys
\begin{equation*}
\bigl| [P^\pm_N e^{\mp it\Delta}](x,y) \bigr|
    \lesssim   \frac{N^d}{(N|x|)^{\frac{d-1}2}\langle N|y|\rangle^{\frac{d-1}2}}
     \bigl\langle N|x| - N|y| \bigr\rangle^{-m}
\end{equation*}
for any $m\geq 0$.
\end{SL}
\end{prop}

%
%
%
%

\section{Localization of kinetic energy}

In this section we prove Theorem~\ref{T:kinetic}.  At the end of this section, we will show how this follows quickly
from the following result.

\begin{prop}\label{P:piece est}
Let $u$ be as in Theorem~\ref{T:kinetic}.  Then
\begin{equation*}
\|\varphi_{> 1} P_N u\|_{L_t^\infty L_x^2([0,\infty)\times\R^d)}\lsm_u N^{-1-\eps}+\|P_Nu_0\|_{L_x^2},
\end{equation*}
for some $\eps=\eps(d)>0$ and any $N\geq 1$.
\end{prop}

\begin{proof}
Fix $t>0$ and $N\geq 1$.  Decomposing $u(t)$ into incoming and outgoing spherical waves and using the Duhamel formula \eqref{duhamel}
into the future for the outgoing spherical waves and into the past for the incoming spherical waves, we write
\begin{align}\label{decomposition}
\varphi_{> 1}P_Nu(t)
&=\varphi_{> 1}P_N^+ u(t)+\varphi_{> 1}P_N^- u(t)  \notag\\
&=i\!\int_0^{\infty}\! \varphi_{> 1} P_N^+ e^{-i\tau\Delta}F\bigl(u(t+\tau)\bigr) \, d\tau + \varphi_{> 1} P_N^- e^{it\Delta} u_0\\
&\quad - i\!\int_0^t \! \varphi_{> 1} P_N^- e^{i\tau\Delta}F\bigl(u(t-\tau)\bigr) \, d\tau, \notag
\end{align}
where $F(u):=-|u|^{\frac 4d}u$ and the first integral is to be understood in the weak topology on $L_x^2$.

By Proposition~\ref{P:P properties} (ii),
\begin{align}\label{pe1}
\|\varphi_{> 1} P_N^- e^{it\Delta} u_0\|_2\lsm \|P_N u_0\|_2,
\end{align}
which is acceptable.

Next, we consider the contribution of the first term on the right-hand side of \eqref{decomposition};
the contribution of the last term can be dealt with by a similar argument.  We will show
\begin{align}\label{pe3}
\left\|\varphi_{> 1}\int_0^{\infty} P_N^+ e^{-i\tau\Delta}F\bigl(u(t+\tau)\bigr) \, d\tau \right\|_2\lsm_u N^{-1-\eps},
\end{align}
which will complete the proof of Proposition~\ref{P:piece est}.  We start by decomposing
\begin{align}\label{decomposition-2}
\int_0^{\infty}\! & P_N^+ e^{-i\tau\Delta}F\bigl(u(t+\tau)\bigr) \, d\tau\\
&= \int_0^{N^{-1}}\!\!\! P_N^+ e^{-i\tau\Delta}\varphi_{\leq \frac 12}F\bigl(u(t+\tau)\bigr) \, d\tau
    + \int_0^{N^{-1}}\!\!\! P_N^+ e^{-i\tau\Delta}\varphi_{> \frac 12}F\bigl(u(t+\tau)\bigr) \, d\tau \notag\\
&\quad + \int_{N^{-1}}^{\infty}\! P_N^+ e^{-i\tau\Delta}\varphi_{\leq \frac{N\tau}2}F\bigl(u(t+\tau)\bigr) \, d\tau
    +\int_{N^{-1}}^{\infty} \!P_N^+ e^{-i\tau\Delta}\varphi_{> \frac{N\tau}2}F\bigl(u(t+\tau)\bigr) \, d\tau,\notag
\end{align}
where the last two integrals are to be understood in the weak topology on $L_x^2$.

To continue, we need bounds on the integrals appearing above; that is the purpose of the next two lemmas.
The first lemma controls the contribution of the `tail' terms on the right-hand side of \eqref{decomposition-2}.

\begin{lem}[The tail]\label{L:tail}
For $N\geq 1$ and $r\in\{\tfrac 14, \tfrac12\}$,
\begin{align*}
\left \|\varphi_{> 2r} \int_0^{N^{-1}} P_N^+ e^{-i\tau\Delta}\varphi_{\leq r}F\bigl(u(t+\tau)\bigr) \, d\tau\right \|_2 &\lsm_u  N^{-10}\\
\left \|\varphi_{> 2r}\int_{N^{-1}}^\infty P_N^+ e^{-i\tau\Delta}\varphi_{\leq rN\tau} F\bigl(u(t+\tau)\bigr) \, d\tau \right \|_2 &\lsm_u N^{-10}.
\end{align*}
\end{lem}

\begin{proof}
Using the kernel estimates in Proposition~\ref{P:P properties}, we immediately deduce
\begin{align*}
\Bigl|\bigl[\varphi_{> 2r} P_N^+ e^{-i\tau\Delta}\varphi_{\leq r}\bigr](x,y)\Bigr|
&\lesssim_m \frac{N^d}{\langle N|x-y|\rangle^m} \varphi_{|x-y|>r} \quad \text{for}\quad 0\leq \tau< N^{-1}\\
\Bigl|\bigl[\varphi_{> 2r} P_N^+ e^{-i\tau\Delta}\varphi_{\leq rN\tau}\bigr](x,y)\Bigr|
&\lesssim_m \frac{N^d}{(N^2\tau)^m\langle N|x-y|\rangle^m} \quad \text{for}\quad \tau \geq N^{-1},
\end{align*}
for any $m\geq 0$.  Notice that for the first inequality, we need to combine parts (iii) and (iv) of Proposition~\ref{P:P properties}.

We will prove the second claim of the lemma; the first follows similarly.  Using the kernel estimates above together with Young's inequality,
\begin{align*}
\biggl\|\varphi_{> 2r}\int_{N^{-1}}^\infty P_N^+ e^{-i\tau\Delta}&\varphi_{\leq rN\tau}F\bigl(u(t+\tau)\bigr) \, d\tau \biggr\|_2\\
& \lesssim \int_{N^{-1}}^\infty \frac{d\tau}{(N^2\tau)^{10}} \Bigl\| \frac{N^d}{\langle N|x|\rangle^{10d}} \ast F(u) \Bigr\|_{L_t^\infty L_x^2}\\
& \lesssim N^{-11} \|F(u)\|_{L^\infty_t L_x^{\frac{2d}{d+4}}}\\
& \lesssim N^{-11} \|u\|_{L_t^\infty L_x^2}^{\frac{d}{d+4}}.
\end{align*}
This finishes the proof of Lemma~\ref{L:tail}.
\end{proof}

The remaining two terms on the right-hand side of \eqref{decomposition-2} will be handled via a bootstrap argument.
The next lemma is the tool that allows us to do this.

\begin{lem}[The main contribution]\label{L:upgrade}
Let $u$ be as in Theorem~\ref{T:kinetic} and $r\in\{\tfrac 14, \tfrac12\}$.  Assume that
\begin{equation}\label{assumption}
\bigl\|\varphi_{>r} P_N u \bigr\|_{L_t^\infty L_x^2([0,\infty)\times\R^d)} \lesssim_u N^{-s}
\end{equation}
for some $0\leq s<1$ and any $N\geq 1$.  Then for any $t\in[0,\infty)$,
\begin{align*}
\left\|\varphi_{>r}\int_0^{N^{-1}} P_N^+ e^{-i\tau\Delta}\varphi_{>r}F\bigl(u(t+\tau)\bigr) \, d\tau \right\|_2
&\lesssim_u N^{-\frac{d-1}d-\frac s{1+s}}\\
\left\|\varphi_{>r}\int_{N^{-1}}^{\infty} P_N^+ e^{-i\tau\Delta}\varphi_{> rN\tau}F\bigl(u(t+\tau)\bigr) \, d\tau\right\|_2
&\lesssim_u  N^{-\frac{d-1}d-\frac s{1+s}}.
\end{align*}
\end{lem}

\begin{proof}
Fix $N\geq 1$ and $t\in[0,\infty)$.  Using the decomposition $u= P_{\leq N^{\frac1{s+1}}}u + P_{> N^{\frac1{s+1}}}u$, we write
\begin{equation}\label{non decom}
F(u)=F\bigl(u_{\leq N^{\frac1{s+1}}}\bigr) + G u_{> N^{\frac1{s+1}}} \quad \text{with} \quad
    G= O\Bigl(|u_{\leq N^{\frac1{s+1}}}|^{\frac 4d}+ |u_{> N^{\frac1{s+1}}}|^{\frac 4d}\Bigr).
\end{equation}

We consider the contribution of the first term on the right-hand side of \eqref{non decom}.
Using the triangle inequality together with the fact that $\varphi_{>r}P^+_N$ is bounded on $L_x^2$, followed by
the weighted Strichartz inequality in Lemma~\ref{L:wes}, Lemma~\ref{L:mismatch_fre}, and Bernstein, we estimate
\begin{align*}
\biggl\|\varphi_{>r}\int_{N^{-1}}^{\infty}& P_N^+ e^{-i\tau\Delta}\varphi_{> rN\tau}F\bigl(u_{\leq N^{\frac1{s+1}}}(t+\tau)\bigr) \, d\tau\biggr\|_2\\
&\leq \biggl\|\varphi_{>r}\int_{N^{-1}}^{\infty} P_N^+ e^{-i\tau\Delta}\varphi_{> rN\tau}P_{\geq \frac N8}F\bigl(u_{\leq N^{\frac1{s+1}}}(t+\tau)\bigr) \, d\tau\biggr\|_2\\
&\quad + \biggl\|\varphi_{>r}\int_{N^{-1}}^{\infty} P_N^+ e^{-i\tau\Delta}\varphi_{> rN\tau}P_{< \frac N8}F\bigl(u_{\leq N^{\frac1{s+1}}}(t+\tau)\bigr) \, d\tau\biggr\|_2\\
&\lesssim \Bigl\||y|^{-\frac{2(d-1)}d}\varphi_{>rN\tau}P_{\geq \frac N8}F\bigl(u_{\leq N^{\frac1{s+1}}}(t+\tau)\bigr)\Bigr\|_{L_{\tau}^{\frac d{d-1}}L_x^{\frac{2d}{d+4}}([N^{-1},\infty)\times \R^d)}\\
&\quad + \bigl\|\tilde P_N \varphi_{> rN\tau}P_{< \frac N8}F\bigl(u_{\leq N^{\frac1{s+1}}}(t+\tau)\bigr) \bigr\|_{L_{\tau}^1L_x^2([N^{-1},\infty)\times\R^d)}\\
&\lesssim \left(\int_{N^{-1}}^{\infty} (N\tau)^{-2}\, d\tau\right)^{\frac{d-1}d} \bigl\|P_{\geq \frac N8}F\bigl(u_{\leq N^{\frac1{s+1}}}\bigr)\bigr\|_{L_t^\infty L_x^{\frac{2d}{d+4}}}\\
&\quad + N^{-10} \int_{N^{-1}}^\infty (N\tau)^{-10}\, d\tau \bigl\|P_{<\frac N4} F\bigl(u_{\leq N^{\frac1{s+1}}}\bigr)\bigr\|_{L_t^\infty L_x^2}\\
&\lesssim N^{-\frac{d-1}d-1} \bigl\| \nabla F\bigl(u_{\leq N^{\frac1{s+1}}}\bigr)\bigr\|_{L_t^\infty L_x^{\frac{2d}{d+4}}}
    + N^{-9} \|u\|_{L_t^\infty L_x^2}^{\frac{d}{d+4}}\\
&\lesssim_u N^{-\frac{d-1}d-\frac s{1+s}}.
\end{align*}
Arguing similarly and using H\"older's inequality in the time variable,
\begin{align*}
\biggl\|\varphi_{>r}\int_0^{N^{-1}}& P_N^+ e^{-i\tau\Delta}\varphi_{> r}F\bigl(u_{\leq N^{\frac1{s+1}}}(t+\tau)\bigr) \, d\tau\biggr\|_2\\
&\lesssim \Bigl\||y|^{-\frac{2(d-1)}d}\varphi_{>r}P_{\geq \frac N8}F\bigl(u_{\leq N^{\frac1{s+1}}}(t+\tau)\bigr)\Bigr\|_{L_{\tau}^{\frac d{d-1}}L_x^{\frac{2d}{d+4}}([0,N^{-1}]\times \R^d)}\\
&\quad + \bigl\|\tilde P_N \varphi_{> r}P_{< \frac N8}F\bigl(u_{\leq N^{\frac1{s+1}}}(t+\tau)\bigr) \bigr\|_{L_{\tau}^1L_x^2([0,N^{-1}]\times\R^d)}\\
&\lesssim N^{-\frac{d-1}d-1} \bigl\| \nabla F\bigl(u_{\leq N^{\frac1{s+1}}}\bigr)\bigr\|_{L_t^\infty L_x^{\frac{2d}{d+4}}}
    + N^{-10} \bigl\|P_{<\frac N4}F\bigl(u_{\leq N^{\frac1{s+1}}}\bigr)\bigr\|_{L_t^\infty L_x^2}\\
&\lesssim_u N^{-\frac{d-1}d-\frac s{1+s}}.
\end{align*}

Finally, we consider the contribution of the second term on the right-hand side of \eqref{non decom}.
Using the fact that $\varphi_{>r}P^+_N$ is bounded on $L_x^2$ together with the weighted Strichartz estimate in
Lemma~\ref{L:wes},
\begin{align*}
\biggl\|\varphi_{>r}\int_{N^{-1}}^{\infty} & P_N^+ e^{-i\tau\Delta}\varphi_{> rN\tau}  \bigl(G u_{> N^{\frac1{s+1}}}\bigr)(t+\tau)\, d\tau\biggr\|_2\\
&\lesssim \Bigl\||y|^{-\frac{2(d-1)}d}\varphi_{>rN\tau} \bigl(G u_{> N^{\frac1{s+1}}}\bigr)(t+\tau)\Bigr\|_{L_{\tau}^{\frac d{d-1}}L_x^{\frac{2d}{d+4}}([N^{-1},\infty)\times \R^d)}\\
&\lesssim \left(\int_{N^{-1}}^{\infty} (N\tau)^{-2}\, d\tau\right)^{\frac{d-1}d} \|u\|_{L_t^\infty L_x^2}^{\frac 4d}\bigl\|\varphi_{>r}u_{> N^{\frac1{s+1}}}\bigr\|_{L_t^\infty L_x^2}\\
&\lesssim_u N^{-\frac{d-1}d-\frac s{1+s}}.
\end{align*}
Note that for $s=0$ the last step uses the finiteness of the mass, while for $s>0$ it uses \eqref{assumption}.
Similarly,
\begin{align*}
\biggl\|\varphi_{>r}\int_0^{N^{-1}}& P_N^+ e^{-i\tau\Delta}\varphi_{>r} \bigl(G u_{> N^{\frac1{s+1}}}\bigr)(t+\tau) \, d\tau\biggr\|_2\\
&\lesssim \Bigl\||y|^{-\frac{2(d-1)}d}\varphi_{>r}  \bigl(G u_{> N^{\frac1{s+1}}}\bigr)(t+\tau)\Bigr\|_{L_{\tau}^{\frac d{d-1}}L_x^{\frac{2d}{d+4}}([0,N^{-1}]\times \R^d)}\\
&\lesssim N^{-\frac{d-1}d} \|u\|_{L_t^\infty L_x^2}^{\frac 4d}\bigl\|\varphi_{>r}u_{> N^{\frac1{s+1}}}\bigr\|_{L_t^\infty L_x^2}\\
&\lesssim_u N^{-\frac{d-1}d-\frac s{1+s}}.
\end{align*}

This completes the proof of Lemma~\ref{L:upgrade}.
\end{proof}

We now return to the proof of Proposition~\ref{P:piece est}.  We remind the reader that it suffices to prove \eqref{pe3}
and that we will do this using the decomposition \eqref{decomposition-2}.  We start by noting that hypothesis \eqref{assumption}
in Lemma~\ref{L:upgrade} holds with $s=0$.  Thus, combining Lemma~\ref{L:tail} with Lemma~\ref{L:upgrade} and taking $r=\tfrac14$, we obtain
$$
\left\|\varphi_{>\frac 12}\int_0^{\infty} P_N^+ e^{-i\tau\Delta}F\bigl(u(t+\tau)\bigr) \, d\tau \right\|_2\lsm_u N^{-\frac{d-1}d}.
$$
As remarked before, an analogous argument can be used to derive
$$
\left\|\varphi_{>\frac 12}\int_0^t P_N^- e^{i\tau\Delta}F\bigl(u(t-\tau)\bigr) \, d\tau \right\|_2\lsm_u N^{-\frac{d-1}d}.
$$
Combining this with the decomposition \eqref{decomposition} and \eqref{pe1}, we obtain
$$
\bigl\|\varphi_{>\frac 12}P_N u(t)\bigr\|_2
\lesssim_u N^{-\frac{d-1}d} + \|P_Nu_0\|_2
\lesssim_u  N^{-\frac{d-1}d} + N^{-1}\|\nabla u_0\|_2
\lesssim_u N^{-\frac{d-1}d},
$$
which shows that \eqref{assumption} holds with $s=\frac{d-1}d$ and $r=\tfrac12$.  Now using this as an input in Lemma~\ref{L:upgrade}
and invoking Lemma~\ref{L:tail} again, we get
$$
\left\|\varphi_{>1}\int_0^{\infty} P_N^+ e^{-i\tau\Delta}F\bigl(u(t+\tau)\bigr) \, d\tau \right\|_2
\lesssim_u N^{-\frac{d-1}d -\frac{d-1}{2d-1}},
$$
which proves \eqref{pe3} with $\eps:= \tfrac{d^2-3d+1}{d(2d-1)}$.  This finishes the proof of Proposition~\ref{P:piece est}.
\end{proof}

\begin{proof}[Proof of Theorem~\ref{T:kinetic}]
Fix $\eta>0$ and $t\in [0,\infty)$.  Let $\eta_1>0$ be a small constant and let $N_0>0$ be a large constant, both to be determined later.
To simplify notation, let $R:=2C(\eta_1) \langle N(t)^{-1}\rangle$, where $C:R^+\to \R^+$ denotes the compactness modulus function associated
to $u$ on $[0,\infty)$.

A simple application of the triangle inequality yields
\begin{align*}
\bigl\|\varphi_{>R} \nabla u(t)\bigr\|_2
\leq \bigl\|P_{\leq N_0}\varphi_{>R} \nabla u(t)\bigr\|_2 + \bigl\|P_{>N_0} \varphi_{>R} \nabla u(t)\bigr\|_2.
\end{align*}

Using the triangle inequality, Bernstein, and the mismatch estimates Lemma~\ref{L:mismatch_real} and Lemma~\ref{L:mismatch_fre} (with $m=2$),
we estimate the low frequencies as follows:
\begin{align*}
\bigl\|P_{\leq N_0}\varphi_{>R} \nabla u(t)\bigr\|_2
&\leq \bigl\|P_{\leq N_0}\varphi_{>R} \nabla P_{\leq 4 N_0} u(t)\bigr\|_2 + \bigl\|P_{\leq N_0}\varphi_{>R} \nabla P_{> 4 N_0} u(t)\bigr\|_2\\
&\lesssim \bigl\|\varphi_{>R} \nabla P_{\leq 4 N_0} \varphi_{\leq \frac R2} u(t)\bigr\|_2
    + \bigl\|\varphi_{>R} \nabla P_{\leq 4 N_0} \varphi_{> \frac R2}u(t)\bigr\|_2\\
&\quad + \sum_{N>4N_0}\bigl\|P_{\leq N_0}\varphi_{>R} \nabla P_N u(t)\bigr\|_2\\
&\lesssim_u N_0^{-1}R^{-2} + N_0\|\varphi_{> \frac R2}u(t)\bigr\|_2 + \sum_{N>4N_0} N^{-1}R^{-2}\\
&\lesssim_u N_0^{-1} + N_0\eta_1.
\end{align*}
In the last step we used the fact that $R\geq 2$.

To estimate the high frequencies, we use the Littlewood-Paley square function estimate together with Lemmas~\ref{L:mismatch_real}
and ~\ref{L:mismatch_fre} (with $m=2$) and Proposition~\ref{P:piece est}:
\begin{align*}
\bigl\|P_{>N_0} & \varphi_{>R} \nabla u(t)\bigr\|_2^2\\
&\lesssim \sum_{N>N_0} \bigl\|P_N \varphi_{>R} \nabla ( P_{<\frac N4} + P_{>4N} ) u(t)\bigr\|_2^2
    + \bigl\|P_N \varphi_{>R} \nabla P_{\frac N4\leq \cdot\leq 4N} u(t)\bigr\|_2^2\\
&\lesssim_u \sum_{N>N_0} N^{-2}R^{-4} + \sum_{N>\frac{N_0}4} \bigl\|\varphi_{>R} \nabla \tilde P_N \varphi_{\leq \frac R2} P_N u(t)\bigr\|_2^2\\
&\qquad \qquad \qquad \qquad  \, + \sum_{N>\frac{N_0}4} \bigl\|\varphi_{>R} \nabla \tilde P_N \varphi_{> \frac R2} P_N u(t)\bigr\|_2^2\\
&\lesssim_u N_0^{-2}R^{-4} + \sum_{N>\frac{N_0}4} N^{-2}R^{-4} + \sum_{N>\frac{N_0}4} N^2 \bigl( N^{-2-2\eps} + \|P_N u(0)\|_2^2 \bigr)\\
&\lesssim_u N_0^{-2} + N_0^{-2\eps} + \bigl\|\nabla P_{>\frac{N_0}4} u(0)\bigr\|_2^2.
\end{align*}

Putting everything together and choosing $N_0=N_0(\eta)$ sufficiently large, and $\eta_1=\eta_1(N_0)$ sufficiently small
finishes the proof of Theorem~\ref{T:kinetic}.
\end{proof}

%
%
%
%

\section{Proof of Theorem~\ref{T:main}}

As discussed in the introduction, we only have to establish the second part of Theorem~\ref{T:main}; the first part is a consequence
of the results in \cite{merle2}.  To this end, let $u$ be a global solution to \eqref{nls} with spherically symmetric initial data
$u(0)=u_0\in H^1_x$ and $M(u)=M(Q)$; assume further that $u$ blows up in the positive time direction in the sense of Definition~\ref{D:blowup}.
Then by Theorem~\ref{T:periodic}, $u$ is almost periodic modulo scaling on $[0,\infty)$ with frequency scale function $N$ and
compactness modulus function $C$.

By the variational characterization of the ground state (see Proposition~\ref{P:variational}), in order to establish the claim
it suffices to prove $E(u)=0$.  We will do so by contradiction.  By the sharp Gagliardo-Nirenberg inequality \eqref{sharp-gn},
we have $E(u)\geq 0$.  Assume therefore that $E(u)>0$; we will derive a contradiction using a truncated version of the virial identity
\eqref{true virial}.

Let $\psi$ be a smooth, radial cutoff such that
\begin{align*}
\psi(r) =
\begin{cases}
1, \quad r \le 1 \\
0, \quad r\geq 2
\end{cases}
\end{align*}
and define the virial function $V:[0,\infty)\to\R^+$ by
\begin{align*}
V_R(t) := \int_{\R^d} |x|^2 \psi\bigl(\tfrac{|x|}R\bigr) |u(t,x)|^2 \,dx,
\end{align*}
where $R$ denotes a radius to be chosen later.  As $u$ has finite mass,
\begin{align}\label{virial bdd}
\bigl|V_R(t)\bigr|\lesssim_u R^2.
\end{align}
A simple computation establishes
\begin{align}\label{virial}
\partial_{tt} V_R(t)
& = \partial_t \left( 2\Im \int \nabla \Bigl[|x|^2 \psi\bigl(\tfrac{|x|}R\bigr)\Bigr] \bar u(t,x) \nabla u(t,x)\, dx \right)\notag\\
&= 16 E(u(t)) +O\left( \frac 1 {R^2} \int_{|x| \ge R} |u(t,x)|^2 \, dx  \right) \\
& \quad + O\left( \int_{|x| \ge R} | \nabla u(t,x) |^2 \, dx+ \int_{|x|\ge R} |u(t,x)|^{\frac{2(d+2)}d} \,dx \right).\notag
\end{align}

In what follows, we distinguish two cases: either $N(t)$ is bounded from below or converges to zero along a subsequence.

\medskip

\noindent{\bf Case I:}  $\inf_{t\in[0,\infty)} N(t) >0 $.
Let $\eta>0$ be a small constant to be chosen later.  In this case, by Theorem~\ref{T:periodic} and Theorem~\ref{T:kinetic},
there exists $R=R(\eta)$ such that
\begin{align}\label{M+E small}
\bigl\|u(t) \bigr\|_{L_x^2(|x|\geq \frac R2)} + \bigl\|\nabla u(t) \bigr\|_{L_x^2(|x|\geq \frac R2)}\leq \eta
\end{align}
for all $t\in[0,\infty)$.  By the Gagliardo-Nirenberg inequality \eqref{sharp-gn}, this also implies localization of the potential energy;
indeed,
\begin{align}\label{pot small}
\bigl\| \varphi_{\ge \frac R2} u(t) \bigr\|_{\frac{2(d+2)}d}^{\frac{2(d+2)}d}
& \lsm \bigl\| \varphi_{\ge \frac R2} u(t) \bigr\|_2^{\frac 4d } \, \bigl\| \nabla \bigl(\varphi_{\ge \frac R2} u(t)\bigr) \bigr\|_2^2 \notag\\
& \lsm \eta^{\frac 4d} \, \Bigl( \bigl\| \varphi_{\ge \frac R2} \nabla u(t) \bigr\|_2^2 + \bigl\| (\nabla \varphi_{\ge \frac R2} ) u(t) \bigr\|_2^2 \Bigr)\notag\\
& \lsm \eta^{\frac 4d} \Bigl( \eta^2 + \tfrac 1 {R^2} \| u(t)\|_2^2 \Bigr) \notag\\
& \lsm_u \eta^{\frac{2(d+2)}d},
\end{align}
provided $R$ is chosen sufficiently large depending on $\eta$.

Inserting \eqref{M+E small} and \eqref{pot small} into \eqref{virial}, and choosing $\eta$ sufficiently small depending on $E(u_0)$
and $R$ sufficiently large depending on $\eta$, we obtain
$$
\partial_{tt}V_R(t)\geq 8 E(u)>0
$$
for all $t\in [0,\infty)$, thus contradicting \eqref{virial bdd}.

\medskip

\noindent{\bf Case II:} $\liminf_{t\to\infty}N(t)=0$.
Let $t_n\nearrow \infty$ such that $N(t_n)\searrow 0$.  Without loss of generality, we may assume
\begin{align}\label{min N}
N(t_n) = \min_{0\le t\le t_n} N(t).
\end{align}
First, we show that along this sequence the kinetic energy remains bounded.

\begin{lem}[Bounded kinetic energy]\label{L:bound k}
In Case II, we have $\|\nabla u(t_n)\|_2\lsm_u 1$ for all $n\ge 1$.
\end{lem}

\begin{proof}
We argue by contradiction. Assume there is a subsequence (which we still denote by $t_n$) such that  $\| \nabla u(t_n) \|_2 \to \infty$
as $n\to \infty$.  Let us define the rescaled functions
\begin{align*}
v_n(x) := \bigl(\tfrac{ \|\nabla Q\|_2}{\|\nabla u(t_n) \|_2} \bigr)^{\frac d2} \, u\bigl(t_n, \tfrac{\|\nabla Q\|_2}{\| \nabla u(t_n)\|_2} x\bigr).
\end{align*}
Clearly,
\begin{equation}\label{equal norm}
\|v_n\|_2 = \|Q\|_2 \quad \text{and} \quad \| \nabla v_n \|_2 = \| \nabla Q\|_2.
\end{equation}
Also,
$$
E(v_n)=\tfrac{\|\nabla Q\|_2^2}{\|\nabla u(t_n)\|_2^2}\,  E(u(t_n))\to 0 \quad \text{as} \quad n\to\infty.
$$
In particular, this implies
\begin{align}\label{pot E convg}
\|v_n\|_{\frac{2(d+2)}d}^{\frac{2(d+2)}d}\to \tfrac {d+2}d\|\nabla Q\|_2^2.
\end{align}

By \eqref{equal norm}, there exists $V\in H_x^1(\R^d)$ such that $v_n$ converge weakly to $V$ in $H_x^1(\R^d)$.  Moreover,
\begin{equation}\label{small}
\| V\|_2\le \liminf _{n\to\infty} \|v_n\|_2=\| Q\|_2 \quad \text{and} \quad \|\nabla V\|_2\le \liminf _{n\to\infty}\|\nabla v_n\|_2=\|\nabla Q\|_2.
\end{equation}
As $u$ is spherically symmetric, it follows that $v_n$ are spherically symmetric; using the compact embedding
$H_{rad}^1(\R^d)\hookrightarrow L_x^{2(d+2)/d}(\R^d)$, we obtain
$$
v_n\to V \text{ strongly in } L_x^{\frac{2(d+2)}d}(\R^d).
$$
Combining this with \eqref{pot E convg}, we get
$$
\|V\|_{\frac{2(d+2)}d}^{\frac{2(d+2)}d}=\tfrac {d+2}d\|\nabla Q\|_2^2.
$$
An application of the sharp Gagliardo-Nirenberg inequality \eqref{sharp-gn} for $V$ together with \eqref{small},
yields $\|\nabla V\|_2= \|\nabla Q\|_2$ and $\|V\|_2=\|Q\|_2$.  As a consequence, $v_n$ converge to $V$ strongly in $H^1_x(\R^d)$.

Collecting the properties of $V$, we find
$$
\|V\|_2=\|Q\|_2, \quad \|\nabla V\|_2=\|\nabla Q\|_2, \quad E(V)=0.
$$
Using the variational characterization of the ground state (see Proposition~\ref{P:variational}) we deduce that $V(x)=e^{i\theta_0}Q(x)$
for some $\theta_0\in [0,2\pi)$.  Thus,
\begin{align} \label{eq_vn_compact}
 \| v_n - e^{i\theta_0} Q \|_{H_x^1} \to 0 \quad \text{as } n\to \infty.
\end{align}

Let $ \lambda_n :=\frac {\| \nabla u(t_n) \|_2} {\| \nabla Q\|_2 N(t_n)}$.  Note that by assumption, $\lambda_n\nearrow\infty$ as $n\to \infty$.
Hence, by \eqref{eq_vn_compact} the rescaled functions $\lambda_n^{\frac d2} v_n(\lambda_n x)$ converge to zero weakly in $L_x^2(\R^d)$.
Decoding what this means for $u$, we find that
$$
N(t_n)^{-\frac d2} u\bigl(t_n, \tfrac 1{N(t_n)} x \bigr)\to 0 \quad \text{weakly in $L_x^2(\R^d)$}.
$$
This clearly contradicts the facts that $M(u)=M(Q)$ and that, by Theorem~\ref{T:periodic},
$\{N(t_n)^{-\frac d2} u(t_n, N(t_n)^{-1} \cdot)\} $ is precompact in $L^2_x(\mathbb R^d)$.

This completes the proof of the lemma.
\end{proof}

\begin{rem}
In the proof of Lemma~\ref{L:bound k}, we used the compactness of the embedding $H_{rad}^1(\R^d)\hookrightarrow
L_x^{2(d+2)/d}(\R^d)$, thus relying on the spherical symmetry of the solution $u$.  Using instead the concentration compactness
result \cite[Theorem~1]{hmidi-keraani}, one can prove Lemma~\ref{L:bound k} in the non-radial case.  Indeed, the whole argument
used to derive \eqref{eq_vn_compact}, which is really just Weinstein's result \eqref{weinstein's result}, is reminiscent of the
techniques used in \cite{hmidi-keraani}.  However, to pass beyond this and prove that finite-time minimal-mass blowup solutions
are the pseudo-conformal ground state up to the symmetries of the equation, the arguments used in \cite{hmidi-keraani, merle1, merle2}
rely heavily on the finiteness of the blowup time.
\end{rem}

Returning to the proof of Theorem~\ref{T:main}, let $\eta>0$ be a small constant to be chosen later.
Using Theorem~\ref{T:periodic} and Theorem~\ref{T:kinetic} and recalling \eqref{min N},
\begin{align*}
\bigl\|u(t) \bigr\|_{L_x^2(|x|\geq \frac {R_n}2)} + \bigl\|\nabla u(t) \bigr\|_{L_x^2(|x|\geq \frac {R_n}2)}\leq \eta,
\end{align*}
for $R_n :=\frac{C(\eta)}{N(t_n)}$ and all $t\in[0, t_n]$.  Arguing as in \eqref{pot small}, this implies
$$
\bigl\| \varphi_{\ge \frac{R_n}2} u(t) \bigr\|_{\frac{2(d+2)}d}^{\frac{2(d+2)}d}
\lsm \eta^{\frac 4d} \Bigl( \eta^2 + \tfrac{N(t_n)^2}{C(\eta)^2} \| u(t)\|_2^2 \Bigr)
\lsm_u \eta^{\frac{2(d+2)}d},
$$
for all $t\in[0, t_n]$.  Thus, by \eqref{virial}, choosing $\eta$ sufficiently small depending on $E(u_0)$,
\begin{align}\label{virial bdd 2}
\partial_{tt}V_{R_n}(t)\geq 8 E(u)>0
\end{align}
for all $t\in [0,t_n]$.

On the other hand, by Lemma~\ref{L:bound k},
$$
\bigl| \partial_t V_{R_n} (0)\bigr| + \bigl| \partial_t V_{R_n} (t_n)\bigr| \lesssim_u R_n.
$$
Thus, using the Fundamental Theorem of Calculus and \eqref{virial bdd 2}, followed by Remark~\ref{R:bdd N}, we obtain
$$
E(u_0) t_n \lesssim_u R_n \lesssim_u N(t_n)^{-1}\lesssim_u t_n^{1/2}.
$$
Letting $n\to \infty$, we reach a contradiction.

This finishes the proof of Theorem~\ref{T:main}.
\qed

\end{document}